\documentclass[12pt,a4paper,reqno]{amsart}
\usepackage{eurosym}

\usepackage{amsmath,amsthm,amscd,amsfonts,amssymb,graphicx,wasysym,soul,enumitem,mathtools,stackrel}
\usepackage{caption}
\usepackage{subcaption}
\usepackage[right=25mm, left=25mm, top=25mm, bottom=25mm]{geometry}

\usepackage[dvipsnames]{xcolor}

\usepackage{tikz}
\usepackage[colorlinks=true,linkcolor=blue,citecolor=red]{hyperref}

\newcommand\redsout{\bgroup\markoverwith{\textcolor{red}{\rule[0.5ex]{2pt}{1.4pt}}}\ULon}

\definecolor{mycyan}{rgb}{0.0, 1.0, 1.0}

\definecolor{mypink}{rgb}{1.0, 0.75, 0.8}

\definecolor{myfuchsiapink}{rgb}{1.0, 0.80, 1.0}

\definecolor{myguppiegreen}{rgb}{0.0, 1.0, 0.5}

\definecolor{mylightcarminepink}{rgb}{0.9, 0.5, 0.9}

\sethlcolor{yellow}

\newtheorem{thm}{Theorem}[section]
\newtheorem*{thrm}{Theorem 1'}

\newtheorem{lem}{Lemma}[section]

\theoremstyle{definition}

\theoremstyle{remark}
\newtheorem{rem}{Remark}[section]

\numberwithin{equation}{section}

\fussy

\begin{document}
\bibliographystyle{amsplain}

\title[Simple proofs of certain results on generalized Fekete-Szeg\H{o} functional]{Simple proofs of certain results on generalized Fekete-Szeg\H{o} functional in the class $\boldsymbol{\mathcal{S}}$}

\author[T. Bulboac\u{a}]{Teodor Bulboac\u{a}}
\address{Faculty of Mathematics and Computers Science, Research Centre of Applied Analysis, Babe\c{s}-Bolyai University, 400084 Cluj-Napoca, Romania.}
\email{bulboaca@math.ubbcluj.ro \& teodor.bulboaca@ubbcluj.ro}

\author[M. Obradovi\'{c}]{Milutin Obradovi\'{c}}
\address{Department of Mathematics,
Faculty of Civil Engineering, University of Belgrade,
Bulevar Kralja Aleksandra 73, 11000, Belgrade, Serbia.}
\email{obrad@grf.bg.ac.rs \& milutin.obradovic@gmail.com}

\author[N. Tuneski]{Nikola Tuneski}
\address{Department of Mathematics and Informatics, Faculty of Mechanical Engineering, Ss. Cyril and
Methodius
University in Skopje, Karpo\v{s} II b.b., 1000 Skopje, Republic of North Macedonia.}
\email{nikola.tuneski@mf.edu.mk \& ntuneski@gmail.com}



\begin{abstract}
In this paper we give simple proofs for the main results concerning generalized Fekete-Szeg\H{o} functional of type $\left|a_{3}(f)-\lambda a_{2}(f)^{2}\right|-\mu|a_{2}(f)|$, where $\lambda\in\mathbb{C}$, $\mu>0$ and $a_{n}(f)$ is $n$-th coefficient of the power series expansion of $f\in\mathcal{S}$. In addition, we studied this functional separately for the class $\mathcal{K}$ of convex functions and we emphasize that all the results of the paper are sharp (i.e. the best possible). The advantages of the present study are that the techniques used in the proofs are more easier and use known results regarding the univalent functions, and those that it give the best possible results not only for the entire class of univalent normalized functions $\mathcal{S}$ but also for its subclass of convex functions $\mathcal{K}$.
\end{abstract}

\maketitle

\section{Introduction and definitions}\label{sectintro}

Let $\mathcal{A}$ be the class of functions $f$ analytic in the open unit disk $\mathbb{D}=\{z\in\mathbb{C}:|z|<1\}$ and normalized such that $f(0)=f^{\prime}(0)-1=0$, that is
\begin{equation}\label{eq 1}
f(z)=z+\sum_{k=2}^{\infty}a_{k}z^{k},\;z\in\mathbb{D},
\end{equation}
and by $\mathcal{S}$ we denote the subclass of $\mathcal{A}$ consisting of univalent functions in the unit disc $\mathbb{D}$.

In \cite{lecko} the authors considered the {\em generalized Fekete-Szeg\H{o} functional} for the class $\mathcal{S}$ of type
\[
F_{\lambda,\mu}(f):=\left|a_{3}(f)-\lambda a_{2}(f)^{2}\right|-\mu|a_{2}(f)|,
\]
where $\lambda\in\mathbb{C}$ and $\mu>0$, and the coefficients $a_{2}(f)=a_{2}$ and $a_{3}(f)=a_{3}$ are given by \eqref{eq 1}. Therefore,
\begin{equation}\label{eq 2}
F_{\lambda,\mu}(f)=\left|a_{3}-\lambda a_{2}^{2}\right|-\mu|a_{2}|,\;\lambda\in\mathbb{C},\,\mu>0.
\end{equation}
We note that for $\mu=0$ and $0\leq\lambda\leq1$ we have the well-known {\em Fekete-Szeg\H{o} inequality} \cite{fekszeg}, that is
\begin{equation}\label{eq 3}
F_{\lambda,0}(f)=\left|a_{3}-\lambda a_{2}^{2}\right|\leq\left\{
\begin{array}{lll}
1+2\exp\left(-\frac{2\lambda}{1-\lambda}\right),&\text{if}&0\leq\lambda<1,\\
1,&\text{if}&\lambda=1.
\end{array}
\right.
\end{equation}
The above inequality is sharp for all $0\leq\lambda\leq1$, in the sense that there exists a function $f_{\lambda}\in\mathcal{S}$ such the equality holds in \eqref{eq 3}.

We emphasize that the {\em Koebe function} $k(z)=z/(1-z)^{2}$ is not the extremal function for the inequality \eqref{eq 3} for $0<\lambda<1$, but it is the extremal function whenever $\lambda\in\{0,1\}$.

The function $\phi:\mathbb{D}\times[0,\infty)\rightarrow\mathbb{C}$ is called a {\em L\"{o}wner chain} (see \cite{lowner}, \cite[p. 157]{pommer}) if $\phi(z,t)=\mathrm{e}^{it}+c_{2}(t)z^{2}+\dots$, $z\in\mathbb{D}$, is analytic in $\mathbb{D}$ for all $t\geq0$ and $\phi(z,s)\prec\phi(z,t)$ for all $0\le s\le t$, where the symbol ``$\prec$'' stands for the subordination. This notion is connected with the well-known {\em L\"{o}wner-Kufarev differential equation} introduced and studied first in \cite{lowner}, then by \cite{kufarev} and further by many authors.

Also, a {\em slit mapping} is a function that maps a domain conformally onto the complex plane minus a set of Jordan curves, while a {\em single-slit mapping} is a slit mapping that's range is the complement of a single Jordan curve.

Two important results for the L\"{o}wner chains in connection with the univalent functions and the single-slit mappings are given by the following properties:

\textsf{A.} If $f\in\mathcal{S}$, then there exists a L\"{o}wner chain $\phi(z,t)$ such that $f(z)=\phi(z,0)$, $z\in\mathbb{D}$ (see \cite[Therorem 6.1., p. 159]{pommer});

\textsf{B.} To each $f\in\mathcal{S}$ there exists a sequence of single-slit mappings $\left(\psi_{n}\right)_{n\in\mathbb{N}}$ such that $\psi_{n}\to f$ ($n\to\infty$) uniformly on each compact subset of $\mathbb{D}$ (see \cite[Theorem 3.2., p. 80]{duren}), that is the density of the slit mappings in the class $\mathcal{S}$ in the space of analytic functions in $\mathbb{D}$ with respect to the topology of uniformly convergence on compact subsets of $\mathbb{D}$.

The main results of \cite{lecko} are obtained by using the above results of the theory of L\"{o}wner chains, and are given in two next theorems.

\begin{thm}\cite[Theorem 1.1.]{lecko}\label{th-1}
For all $f\in\mathcal{S}$, $\lambda\in\mathbb{C}$, $\mu>0$, and $F_{\lambda,\mu}(f)$ is defined by \eqref{eq 2} we have
\begin{equation}\label{eq 4}
F_{\lambda,\mu}(f)\leq
\left\{\begin{array}{lll}
1,&\text{if}&|1-\lambda|\leq\frac{\mu}{2},\\
4|1-\lambda|-2\mu+1,&\text{if}&|1-\lambda|>\frac{\mu}{2}.
\end{array}
\right.
\end{equation}
Moreover, for each $f:\mathbb{D}\rightarrow\mathbb{C}$ the following implications hold:
\begin{itemize}

\item [(i)] If $|1-\lambda|\leq\frac{\mu}{2}$ and $\lambda\neq 1+\frac{\mu}{2}$, then $f\in\mathcal{S}$ and the equality in \eqref{eq 4} holds if and only if there exists $|\zeta|=1$ such that
\begin{equation}\label{eq 5}
f(z)=\frac{z}{1+\zeta^{2}z^{2}},\;z\in\mathbb{D}.
\end{equation}

\item [(ii)] If $\lambda=1+\frac{\mu}{2}$, then $f\in\mathcal{S}$ and the equality in \eqref{eq 4} holds if and only if $f$ is of the form \eqref{eq 5} or $f$ is a rotation of the Koebe function $k$;

\item [(iii)] If $|1-\lambda|>\frac{\mu}{2}$, then $f\in\mathcal{S}$ and the equality in \eqref{eq 4} holds if and only if $\lambda>1+\frac{\mu}{2}$ and $f$ is a rotation of the Koebe function $k$.
\end{itemize}
\end{thm}

\begin{thm}\cite[Theorem 1.2.]{lecko}\label{th-2}
For all $f\in\mathcal{S}$, $\lambda\in\mathbb{C}\setminus\{1\}$, $\mu>0$, and $F_{\lambda,\mu}(f)$ is defined by \eqref{eq 2} we have
\begin{equation}\label{eq 6}
F_{\lambda,\mu}(f)\geq
\left\{\begin{array}{lll}
-\frac{\mu}{\sqrt{|1-\lambda|}},&\text{if}&|1-\lambda|>\frac{\mu^{2}}{4},\\
-2\mu,&\text{if}&|1-\lambda|\leq\frac{\mu^{2}}{4}.
\end{array}
\right.
\end{equation}
Moreover, for each $f:\mathbb{D}\rightarrow\mathbb{C}$ the following implications hold:
\begin{itemize}

\item[(i)] If $|1-\lambda|>\frac{\mu^{2}}{4}$, then $f\in\mathcal{S}$ and the equality in \eqref{eq 6} holds if and only if $\lambda\leq\frac{3}{4}$ and $\lambda<1-\frac{\mu^{2}}{4}$, and there exists $|\zeta|=1$ such that
\begin{equation}\label{eq 7}
f(z)=\frac{z}{1-\frac{1}{\sqrt{|1-\lambda|}}\zeta z+\zeta^{2}z^{2}},\;z\in\mathbb{D}.
\end{equation}

\item[(ii)] If $|1-\lambda|\leq\frac{\mu^{2}}{4}$, then $f\in\mathcal{S}$ and the equality in \eqref{eq 6} holds if and only if $\lambda=\frac{3}{4}$ and $\mu\geq1$, and $f$ is a rotation of the Koebe function $k$.
\end{itemize}
\end{thm}

The main aim of this paper is to give simple proofs of the two previous theorems, and additionally, in the last section there are given sharp estimates of $F_{\lambda,\mu}(f)$ for the class of convex functions.

\section{Simple proofs of Theorem \ref{th-1} and Theorem \ref{th-2}}\label{simpleproofs}

For our proofs we will mainly use the next lemma.

\begin{lem}\cite{bieber}, \cite[Theorem 2, p. 35]{goodman}\label{lem-1}
For every function $f\in\mathcal{S}$ given by \eqref{eq 1} we have
\begin{equation}\label{eq 8}
\left|a_{3}-a_{2}^{2}\right|\leq 1,
\end{equation}
and the equality in \eqref{eq 8} holds if and only if
\begin{equation}\label{eq 9}
f(z)=\frac{z}{1-b\zeta z+\zeta^{2}z^{2}},\;z\in\mathbb{D},\;|b|\le2,\;|\zeta|=1.
\end{equation}
\end{lem}

We would like to emphasize that the range of $\mathbb{D}$ by the above function $f$ was determined in \cite[Lemma 2.3, p. 695]{lecko-part23} and it's connected with the slit mappings, while the fact that the equality in \eqref{eq 8} holds if and only if $f$ has the form \eqref{eq 9} was also shown in \cite[Lemma 2.4, p. 696]{lecko-part23}.

As an immediately consequence of the above result, if $f\in\mathcal{S}$ is an odd function of the form \eqref{eq 1}, then $|a_{3}|\le1$ and the equality holds for the function $f$ given by \eqref{eq 9}.

\begin{proof}[Proof of Theorem \ref{th-1}]
Using Lemma \ref{lem-1} and \eqref{eq 2} we obtain
\begin{align}
F_{\lambda,\mu}(f)&=\left|\left(a_{3}-a_{2}^{2}\right)+(1-\lambda)a_{2}^{2}\right|-\mu|a_{2}|\nonumber\\
&\leq\left|a_{3}-a_{2}^{2}\right|+|1-\lambda||a_{2}|^{2}-\mu|a_{2}|\label{eq 10}\\
&\leq|1-\lambda||a_{2}|^{2}-\mu|a_{2}|+1=:\varphi\left(|a_{2}|\right),\label{eq 10i}
\end{align}
where (see also \cite{bieber})
\[
\varphi(t)=|1-\lambda|t^{2}-\mu t+1,\;t:=|a_{2}|\in[0,2].
\]

If $\lambda=1$, then $F_{\lambda,\mu}(f)=\left|a_{3}-a_2^2\right|-\mu|a_2|$ which, due to the inequality \eqref{eq 8} of Lemma \ref{lem-1} and $\mu>0$ implies $F_{\lambda,\mu}(f)\le1-\mu|a_2|$. It is easy to verify that for the univalent function $f$ given by \eqref{eq 5} we have
\[
f(z)=z-\zeta^{2}z^{3}+\dots,\;z\in\mathbb{D},
\]
thus for this function we obtain equality in the estimate.

On the other hand, if $F_{\lambda,\mu}(f)=1$ then it follows that $|a_{3}-a_2^2|=1$ and $|a_2|=0$, which according to the Lemma \ref{lem-1} is satisfied if and only if the function $f$ has the form \eqref{eq 9} with $a_{2}=0$, that is $f$ is given by \eqref{eq 5}.

If $\lambda\neq1$, then we have
\begin{equation}\label{eq 12}
F_{\lambda,\mu}(f)\le\max\left\{\varphi(0);\varphi(2)\right\}=\left\{
\begin{array}{lll}
1,&\text{if}&|1-\lambda|\leq\frac{\mu}{2},\\
4|1-\lambda|-2\mu+1,&\text{if}&|1-\lambda|\geq\frac{\mu}{2}.
\end{array}
\right.
\end{equation}

It is easy to check that the estimates \eqref{eq 4}, that are in fact \eqref{eq 12}, are reached (i.e. equalities sign hold) for the functions given in $(i)$, $(ii)$ and $(iii)$ of Theorem \ref{th-1}. Next we will show that those are the only such functions.

Let assume that equality holds in the inequality \eqref{eq 12}. This means that equalities hold in each of the inequalities \eqref{eq 10} and \eqref{eq 10i}.

\textsf{(a)} For the case $|1-\lambda|\leq\frac{\mu}{2}$ we have the equality $F_{\lambda,\mu}(f)=1$ whenever $1=\varphi(0)\ge\varphi(2)$, while $\varphi(0)=1$ if and only if $|a_{2}|=0$, hence the function $f$ given by \eqref{eq 5} satisfies the equality in \eqref{eq 12}.

\textsf{(b)} In the case $|1-\lambda|\geq\frac{\mu}{2}$ we get $F_{\lambda,\mu}(f)\leq4|1-\lambda|-2\mu+1$. Since $\varphi(t)\le\varphi(2)$ for all $t\in[0,2]$, the equality in \eqref{eq 10i} holds if and only if $|a_{2}|=2$. According to the well-known Bieberbach's result \cite{bieber} (see also \cite[Theorem 1, p. 33]{goodman}) we have $|a_{2}|=2$ if and only if $f$ is a rotation of the Koebe function, that is
\begin{equation}\label{formKoebe}
k_{\theta}(z)=\frac{z}{\left(1-\mathrm{e}^{i\theta}z\right)^{2}}=z+2\mathrm{e}^{i\theta}z^{2}+3\mathrm{e}^{2i\theta}z+\dots,\;z\in\mathbb{D},
\;\theta\in\mathbb{R}.
\end{equation}
Therefore, the condition that $f=k_{\theta}$ is a necessary one to have $F_{\lambda,\mu}(f)=4|1-\lambda|-2\mu+1$. It remains to prove that this is also a sufficient condition for having the previous equality.

Thus, if $f=k_{\theta}$, then $a_{2}=2\mathrm{e}^{i\theta}$ and $a_{3}=3\mathrm{e}^{2i\theta}$, $\theta\in\mathbb{R}$, that implies $\left|a_{3}-a_{2}^{2}\right|=1$, hence we get equality in \eqref{eq 10i}. In addition, the inequality \eqref{eq 10} will be equivalent to
\[
\left|-\mathrm{e}^{2i\theta}+4(1-\lambda)\mathrm{e}^{2i\theta}\right|-2\mu\le 4|1-\lambda|-2\mu+1
\]
or
\begin{equation}\label{eqine23}
\left|4(\lambda-1)+1\right|\le 4|\lambda-1|+1.
\end{equation}

It is obvious that for all $z_{1},z_{2}\in\mathbb{C}$ we have $\left|z_{1}+z_{2}\right|=|z_{1}|+|z_{2}|$ if and only if there exists a number $s\ge0$ such that $z_{1}=cz_{2}$. Denoting $z_{1}:=4(\lambda-1)$ and $z_{2}:=1$, it follows that we obtain equality in \eqref{eqine23} if and only if there exists $s\ge0$ for which $4(\lambda-1)=c$, that's equivalent to $\lambda\ge1$.

From the above reasons we conclude that if $|1-\lambda|\geq\frac{\mu}{2}$, then $F_{\lambda,\mu}(f)=4|1-\lambda|-2\mu+1$ if and only if $f=k_{\theta}$ and $\lambda>1$. We mention that is this last inequality we considered a strictly one because we already assumed that $\lambda\neq1$. Using the fact that $\lambda>1$, that is $\lambda$ should be a real number, we finally obtain that if $\lambda\geq1+\frac{\mu}{2}$, then $F_{\lambda,\mu}(f)=4\lambda-3-2\mu$ if and only if $f=k_{\theta}$.
\end{proof}

According to the above proof, the Theorem \ref{th-1} could be reformulated in the below slightly changed form:

\begin{thrm}\label{thrm1prime}
For all $f\in\mathcal{S}$, $\lambda\in\mathbb{C}$, $\mu>0$, and $F_{\lambda,\mu}(f)$ defined by \eqref{eq 2} we have
\begin{equation}\label{eq 4new}
F_{\lambda,\mu}(f)\leq
\left\{\begin{array}{lll}
1,&\text{if}&|1-\lambda|\leq\frac{\mu}{2},\\
4|1-\lambda|-2\mu+1,&\text{if}&|1-\lambda|\ge\frac{\mu}{2}.
\end{array}
\right.
\end{equation}
Moreover, for each $f:\mathbb{D}\rightarrow\mathbb{C}$ the following implications hold:
\begin{itemize}
\item [(i)] If $|1-\lambda|\leq\frac{\mu}{2}$, then $f\in\mathcal{S}$ and the equality in \eqref{eq 4new} holds if and only if there exists $|\zeta|=1$ such that
\[
f(z)=\frac{z}{1+\zeta^{2}z^{2}},\;z\in\mathbb{D}.
\]

\item [(ii)] If $|1-\lambda|\ge\frac{\mu}{2}$, then $f\in\mathcal{S}$ and the equality in \eqref{eq 4new} holds if and only if $1-\lambda<-\frac{\mu}{2}$ and $f$ is a rotation of the Koebe function $k$.
\end{itemize}
\end{thrm}

\begin{rem}\label{remfirst}
Since $|a_{2}|\leq2$ for all $f\in\mathcal{S}$, then $\frac{1}{2}|a_{2}|\leq1$. From the definition formula \eqref{eq 2}, using the triangle's inequality together with the Fekete-Szeg\H{o} theorem \cite{fekszeg} given by \eqref{eq 3} we obtain
\begin{align}
F_{\lambda,\mu}(f)&=\left|a_{3}-\lambda a_{2}^{2}\right|-\mu|a_{2}|\leq\left|a_{3}-\lambda a_{2}^{2}\right|-\mu|a_{2}|\cdot\frac{1}{2}|a_{2}|\nonumber\\
&\leq\left|a_{3}-\left(\lambda+\frac{\mu}{2}\right)a_{2}^{2}\right|\nonumber\\
&\leq\left\{
\begin{array}{lll}
1+2\exp\left(-\frac{2(\lambda+\mu/2)}{1-(\lambda+\mu/2)}\right),&\text{if}&0\leq\lambda+\frac{\mu}{2}<1,\label{eq 13}\\
1,&\text{if}&\lambda+\frac{\mu}{2}=1.
\end{array}
\right.
\end{align}
We emphasize that the estimate \eqref{eq 13} is better than the one that follows from Theorem \ref{th-1} in the case $0\leq\lambda+\frac{\mu}{2}<1$, like we will show below.

Indeed, having in mind that the assumption $0\leq\lambda+\frac{\mu}{2}<1$ implies $1-\lambda>\frac{\mu}{2}$, the inequality
\[
1+2\exp\left(-\frac{2(\lambda+\mu/2)}{1-(\lambda+\mu/2)}\right)<4|1-\lambda|-2\mu+1
\]
is equivalent to
\[
\exp\left(-\frac{2(\lambda+\mu/2)}{1-(\lambda+\mu/2)}\right)<2\left[1-\left(\lambda+\frac{\mu}{2}\right)\right],
\]
which after substitution $1-\left(\lambda+\frac{\mu}{2}\right)=\frac{1}{t}$, $t\geq 1$ \big(since $0\leq\lambda+\frac{\mu}{2}<1$\big), becomes
\[
\frac{2 e^{2t}}{e^{2}}-t>0,\;t\in[1,+\infty).
\]
According to the first derivative test it's easy to check that the above inequality holds that proves the required conclusion.
\end{rem}

\begin{proof}[Proof of Theorem \ref{th-2}]
In the beginning, let note that for every $f\in\mathcal{S}$ we have $|a_{2}|\leq2$ (see \cite{bieber} or \cite[Theorem 1, p. 33]{goodman}), hence
\begin{equation}\label{eq-14}
F_{\lambda,\mu}(f)=\left|a_{3}-\lambda a_{2}^{2}\right|-\mu|a_{2}|\geq-\mu|a_{2}|\ge-2\mu,
\end{equation}
which proves that the inequality \eqref{eq-14} holds for all $f\in\mathcal{S}$.

It's obvious that for the case $|1-\lambda|\leq\frac{\mu^{2}}{4}$, if $\lambda=\frac{3}{4}$ and $\mu\geq1$ and $f$ is a rotation of the Koebe function $k$, then $f\in\mathcal{S}$ and by using the power series expansion \eqref{formKoebe} we get the equality $F_{\lambda,\mu}(f)=-2\mu$.

Reversely, a sufficient condition to have equality in \eqref{eq 6}, more exactly in \eqref{eq-14}, is that $|a_{2}|=2$. From the previously used Bieberbach's result, $|a_{2}|=2$ if and only if $f$ is rotation of the Koebe function, that is $f=k_{\theta}$ where $k_{\theta}$ is given by \eqref{formKoebe}. Thus, if $f=k_{\theta}$ then
\[
F_{\lambda,\mu}(f)=\left|a_{3}-\lambda a_{2}^{2}\right|-\mu|a_{2}|=|4\lambda-3|-2\mu,
\]
so $F_{\lambda,\mu}(f)=-2\mu$ is equivalent to $\lambda=\frac{3}{4}$. This last inequality combined with $|1-\lambda|\le\frac{\mu^{2}}{4}$ lead us to $\mu\ge1$, thus the proof of the item $(ii)$ is complete.

Let's now study the case $|1-\lambda|\geq\frac{\mu^{2}}{4}$. Therefore, the inequality we should prove is the first part of \eqref{eq 6} that is equivalent to
\begin{equation}\label{eq 15}
\left|a_{3}-\lambda a_{2}^{2}\right|\geq\mu\left(|a_{2}|-\frac{1}{\sqrt{|1-\lambda|}}\right).
\end{equation}

If $|a_{2}|\le\frac{1}{\sqrt{|1-\lambda|}}$, then \eqref{eq 15} obviously holds. Considering the case $\frac{1}{\sqrt{|1-\lambda|}}\leq|a_{2}|\leq2$, from the well-known inequality for complex numbers $\left|z_{1}+z_{2}\right|\geq\big|\left|z_{1}\right|-\left|z_{2}\right|\big|$ and Lemma \ref{lem-1} we get
\begin{align}
\left|a_{3}-\lambda a_{2}^{2}\right|&=\left|\left(a_{3}-a_{2}^{2}\right)+(1-\lambda)a_{2}^{2}\right|
\geq|1-\lambda||a_{2}|^{2}-\left|a_{3}-a_{2}^{2}\right|\nonumber\\
&\geq|1-\lambda||a_{2}|^{2}-1.\label{ineqi}
\end{align}
Consequently, to prove the inequality \eqref{eq 15} it is enough to show that
\begin{equation}\label{eqninis}
|1-\lambda||a_{2}|^{2}-1\geq\mu\left(|a_{2}|-\frac{1}{\sqrt{|1-\lambda|}}\right),
\end{equation}
or equivalently,
\[
|1-\lambda|\left(|a_{2}|^{2}-\frac{1}{|1-\lambda|}\right)-\mu\left(|a_{2}|-\frac{1}{\sqrt{|1-\lambda|}}\right)\geq0,
\]
that is
\begin{equation}\label{eq 18}
\left(|a_{2}|-\frac{1}{\sqrt{|1-\lambda|}}\right)\cdot
\left[|1-\lambda|\left(|a_{2}|+\frac{1}{\sqrt{|1-\lambda|}}\right)-\mu\right]\geq0.
\end{equation}
The inequality \eqref{eq 18} is true since by assumption $|a_{2}|-\frac{1}{\sqrt{|1-\lambda|}}\geq0$, and
\begin{equation}\label{instrict}
|1-\lambda|\left(|a_{2}|+\frac{1}{\sqrt{|1-\lambda|}}\right)-\mu\geq|1-\lambda|\frac{2}{\sqrt{|1-\lambda|}}-\mu
=2\left(\sqrt{|1-\lambda|}-\frac{\mu}{2}\right)>0.
\end{equation}

Now we will prove the item $(i)$, that is if $|1-\lambda|>\frac{\mu^{2}}{4}$, then $f\in\mathcal{S}$ and the equality in \eqref{eq 6} holds if and only if $\lambda\leq\frac{3}{4}$ and $\lambda<1-\frac{\mu^{2}}{4}$, and there exists $|\zeta|=1$ such that $f$ has the form \eqref{eq 7}.

\textsf{(a)} If $f$ has the form \eqref{eq 7}, then
\[
a_{2}=\frac{\zeta}{\sqrt{{|\lambda-1|}}},\;a_{3}=\zeta^{2}\left(\frac{1}{{|\lambda-1|}}-1\right),\;|\zeta|=1.
\]
It follows that
\[
F_{\lambda,\mu}(f)=\left|a_{3}-\lambda a_{2}^{2}\right|-\mu|a_{2}|=\frac{\big||\lambda-1|+\lambda-1|\big|}{|\lambda-1|}-\frac{\mu}{\sqrt{|\lambda-1|}},
\;\lambda\in\mathbb{C}\setminus\{1\},
\]
hence
\begin{gather}
F_{\lambda,\mu}(f)=\frac{\big||\lambda-1|+\lambda-1|\big|}{|\lambda-1|}-\frac{\mu}{\sqrt{|\lambda-1|}}=-\frac{\mu}{\sqrt{|\lambda-1|}}\Leftrightarrow
|\lambda-1|+\lambda-1=0\Leftrightarrow\nonumber\\
1-\lambda=|1-\lambda|>0,\;\lambda\in\mathbb{C}\setminus\{1\}\Leftrightarrow\lambda<1.\label{eqequa}
\end{gather}

It's easy to check that the function $f$ defined by \eqref{eq 9} is univalent in $\mathbb{D}$ whenever $|b|\le2$, but if $|b|>2$ this function is not analytic in $\mathbb{D}$ because it has at least one pole in $\mathbb{D}$. Therefore, for the function $f$ given by \eqref{eq 7} we deduce that it is univalent in $\mathbb{D}$ if and only if
\[
\frac{1}{\sqrt{|1-\lambda|}}\le2,
\]
and using \eqref{eqequa} the above inequality is equivalent to $\lambda\le\frac{3}{4}$. Consequently, the assumption $|1-\lambda|>\frac{\mu^{2}}{4}$ becomes $\lambda<1-\frac{\mu^{2}}{4}$, and one direction of the equivalence is proved.

\textsf{(b)} Assume that $f\in\mathcal{S}$ and the equality in \eqref{eq 6} holds. A necessary and sufficient condition for having this equality is those for which we have equality in \eqref{eqninis}. According to \eqref{instrict} and \eqref{eq 18}, the equality in \eqref{eqninis} holds if and only if $|a_{2}|=\frac{1}{\sqrt{|1-\lambda|}}$.

Since the equality holds in \eqref{ineqi} in and only if $\left|a_{3}-a_{2}^{2}\right|=1$ and using Lemma \ref{lem-1} this implies that $f$ has the form \eqref{eq 9}, consequently $f$ will have the form \eqref{eq 7}. Using the previous reasons we deduce $\lambda\le\frac{3}{4}$ and $\lambda<1-\frac{\mu^{2}}{4}$ and the fact that this function gives the required equality.
\end{proof}

\begin{rem}\label{remthm2}
From the above proof of the Theorem \ref{th-2} we get that for all $f\in\mathcal{S}$, $\lambda\in\mathbb{C}$, and $\mu>0$ we have
\[
F_{\lambda,\mu}(f)\geq-2\mu,
\]
while if $\lambda\in\mathbb{C}\setminus\{0\}$ and $|1-\lambda|\geq\frac{\mu^{2}}{4}$, then
\[
F_{\lambda,\mu}(f)\geq-\frac{\mu}{\sqrt{|1-\lambda|}}.
\]
Comparing the above two lower bounds for $F_{\lambda,\mu}(f)$, since
\[
\max\left\{-2\mu;-\frac{\mu}{\sqrt{|1-\lambda|}}\right\}=\left\{
\begin{array}{lll}
-\frac{\mu}{\sqrt{|1-\lambda|}},&\text{if}&|1-\lambda|\ge\frac{1}{4},\\
-2\mu,&\text{if}&|1-\lambda|\le\frac{1}{4},
\end{array}
\right.
\]
we conclude that for all $f\in\mathcal{S}$, $\lambda\in\mathbb{C}$, and $\mu>0$, the next inequality holds
\[
F_{\lambda,\mu}(f)\geq
\left\{\begin{array}{lll}
-\frac{\mu}{\sqrt{|1-\lambda|}},&\text{if}&|1-\lambda|\ge\max\left\{\frac{\mu^{2}}{4};\frac{1}{4}\right\}\\
-2\mu,&&\text{otherwise},
\end{array}
\right.
\]
and gives a more comprehensive estimation for $F_{\lambda,\mu}(f)$ than \eqref{eq 4new}.
\end{rem}

\section{The corresponding estimations for the class $\mathcal{K}$}\label{sectclassK}

It's well-known that an analytic function in $\mathbb{D}$ is called to be a {\em convex function} if it is univalent in $\mathbb{D}$ and the range $f(\mathbb{D})$ is a convex domain (see \cite{study}, \cite[p. 44]{pommer}). We denote by $\mathcal{K}$ the class of convex and normalized functions $f$ by $f(0)=f^{\prime}(0)-1=0$, hence it's evident that $\mathcal{K}\subsetneq\mathcal{S}$, and the analytic equivalent characterization of the class of convex functions $f$ first stated in \cite{study} is $f^{\prime}(0)\ne0$ and $\operatorname{Re}\left(1+\frac{zf^{\prime\prime}(z)}{f^{\prime}(z)}\right)>0$ in $\mathbb{D}$ (see, also \cite[Theorem 1, p. 111]{goodman}, \cite[Theorem 2.7., p. 44]{pommer}). More information about these classes can be found in \cite{pommer, duren}, etc.

To obtain the similar results for the class $\mathcal{K}$ we need the next lemma given by Trimble in \cite{trimble}.

\begin{lem}\label{lem-2}
For all functions $f(z)=z+a_2z^2+a_{3}z^3+\dots$ from $\mathcal{K}$,
\begin{equation}\label{eq 19}
\left|a_{3}-a_2^{2}\right|\leq\frac{1}{3}\left(1-|a_2|^2\right).
\end{equation}
The inequality is sharp with extremal function
\begin{equation}\label{eq 20}
f_{\alpha}(z)=\int_0^z\left(\frac{1+t}{1-t}\right)^{\alpha}\frac{1}{1-t^2}\,\mathrm{d}t
=z+\alpha z^2+\frac{1}{3}\left(2\alpha^2+1\right)z^3+\dots,\;z\in\mathbb{D},
\end{equation}
and $0\le\alpha\leq1$.
\end{lem}

Note that the range of $\alpha$ for the extremal function exploits the fact that if $f\in\mathcal{K}$ then $|a_{n}|\leq 1$, $n=2,3,\ldots$.

\begin{thm}\label{thmK}
For all $f\in\mathcal{K}$, $\lambda\in\mathbb{C}$, $\mu>0$ and $F_{\lambda,\mu}(f)$ is defined by \eqref{eq 2}, we have:
\begin{itemize}

\item[($i$)] If $|1-\lambda|\leq\frac{1}{3}$, then $F_{\lambda,\mu}(f)\leq\frac{1}{3}$. The equality holds if and only if $f$ has the form given in \eqref{eq 20} for $\alpha=0$, i.e. for the function $f_{0}(z)=z+\frac{1}{3}z^{3}+\dots$, $z\in\mathbb{D}$.

\item[($ii$)] If $|1-\lambda|>\frac{1}{3}$, then
\[
F_{\lambda,\mu}(f)\leq\left\{
\begin{array}{lll}
\frac{1}{3},&\text{if}&\frac{1}{3}<|1-\lambda|\leq\frac{1}{3}+\mu,\\
|1-\lambda|-\mu,&\text{if}&\frac{1}{3}+\mu\geq|1-\lambda|.
\end{array}
\right.
\]
Both of these estimations are sharp, since the equality holds in the first case for $a_2=0$ and for $a_2=1$ in the second case, i.e. for the functions $f_{0}$ and $f_{1}$ defined by \eqref{eq 20}, respectively.
\end{itemize}
\end{thm}

\begin{proof}
Similarly as in the proof of Theorem \ref{th-1}, form the triangle's inequality and using the inequality \eqref{eq 19} of Lemma \ref{lem-2} we have
\begin{align}
F_{\lambda,\mu}(f)&=\left|\left(a_{3}-a_{2}^{2}\right)+(1-\lambda)a_{2}^{2}\right|-\mu|a_{2}|\leq
\left|a_{3}-a_{2}^{2}\right|+|1-\lambda||a_{2}|^{2}-\mu|a_{2}|\nonumber\\
&\leq\frac{1}{3}\left(1-|a_{2}|^{2}\right)+|1-\lambda||a_{2}|^{2}-\mu|a_{2}|\label{inqdfk}\\
&=\left(|1-\lambda|-\frac{1}{3}\right)|a_{2}|^{2}-\mu|a_{2}|+\frac{1}{3}=:\psi_{1}\left(|a_{2}|\right),\label{eq 21}
\end{align}
where
\begin{equation}\label{eq 22}
\psi_{1}(t)=\left(|1-\lambda|-\frac{1}{3}\right)t^{2}-\mu t+\frac{1}{3},\;0\leq t\leq1,
\end{equation}
and we used that $|a_2|\le1$ for $f\in\mathcal{K}$ (see \cite{lowner}, \cite[Theorem 7, p. 117]{goodman}).

If $|1-\lambda|\le\frac{1}{3}$, since $|1-\lambda|-\frac{1}{3}\le0$ and $-\mu<0$, from \eqref{eq 21} and \eqref{eq 22} it follows that $F_{\lambda,\mu}(f)\leq\psi_{1}\left(|a_{2}|\right)\leq\psi_{1}(0)=\frac{1}{3}$. The last of the previous inequalities becomes equality if and only if $|a_{2}|=0$, while in view of Lemma \ref{lem-2} the inequality \eqref{inqdfk} becomes equality if and only if the $f=f_{\alpha}$ defined by \eqref{eq 20}, therefore
\[
f(z)=f_{0}(z)=\int_{0}^{z}\frac{1}{1-t^{2}}\,\mathrm{d}t=z+\frac{1}{3}z^{3}+\sum_{k=2}^{\infty}\frac{z^{2k+1}}{2k+1},\;z\in\mathbb{D},
\]
hence $F_{\lambda,\mu}(f)=\frac{1}{3}$ if and only if $f=f_{0}$.

If $|1-\lambda|>\frac{1}{3}$, then since $\psi_{1}(0)=\frac{1}{3}$ and $\psi_{1}(1)=|1-\lambda|-\mu$, we have
\[
F_{\lambda,\mu}(f)\leq\max\left\{\frac{1}{3};|1-\lambda|-\mu\right\}
=\left\{\begin{array}{lll}
\frac{1}{3},&\text{if}&\frac{1}{3}<|1-\lambda|\leq\frac{1}{3}+\mu,\\
|1-\lambda|-\mu,&\text{if}&\frac{1}{3}+\mu\geq|1-\lambda|.
\end{array}
\right.
\]
Using the above power series expansion of $f_{0}$ and the fact that
\[
f_{1}(z)=\int_{0}^{z}\frac{1}{(1-t)^{2}}\,\mathrm{d}t=z+\sum_{k=2}^{\infty}z^{k},\;z\in\mathbb{D},
\]
it follows that the equality holds if the above two cases for the functions $f_{0}$ and $f_{1}$ defined \eqref{eq 20}, respectively.
\end{proof}

\begin{thm}\label{thrmK2}
Let $F_{\lambda,\mu}(f)$ is defined by \eqref{eq 2}, $\lambda\in\mathbb{C}$, and $\mu>\frac23$. Then, for every function $f\in\mathcal{K}$ we have
\begin{equation}\label{estimthmK2}
F_{\lambda,\mu}(f)\geq\left\{\begin{array}{lll}
-\frac{1}{3}-\frac{\mu^{2}}{4\left(|1-\lambda|+\frac{1}{3}\right)},&\text{if}&|1-\lambda|\geq-\frac{1}{3}+\frac{\mu}{2},\\
|1-\lambda|-\mu,&\text{if}&|1-\lambda|\leq-\frac{1}{3}+\frac{\mu}{2}.
\end{array}
\right.
\end{equation}
Both of these results are sharp, as follows:

$(i)$ The equality holds in the first case for the function $f_{t_{0}}$ defined by \eqref{eq 20} for $\alpha=t_{0}$ where $t_{0}=\frac{\mu}{2\left(|1-\lambda|+\frac{1}{3}\right)}$, under the conditions
$\frac{1}{3}-\frac{3}{4}\mu^{2}\leq 1-\lambda\leq\frac{1}{3}-\frac{1}{2}\mu$, $\mu>\frac{2}{3}$.

$(ii)$ The equality holds in the second case for the function $f_{1}$ defined by \eqref{eq 20} for $\alpha=1$ or for its rotations.
\end{thm}

\begin{proof}
In the beginning, let note that for every $f\in\mathcal{K}$, since $|a_{2}|\le1$ we get
\[
F_{\lambda,\mu}(f)=\left|a_{3}-\lambda a_{2}^{2}\right|-\mu|a_{2}|\ge-\mu,
\]
that is $F_{\lambda,\mu}(f)\geq-\mu$, which is equivalent to $\left|a_{3}-\lambda a_{2}^{2}\right|\geq\mu\left(|a_{2}|-1\right)$. For $\lambda=1$ the equality is obtained if $f(z)=\frac{z}{1-z}$ and its rotations.

For $\lambda\neq1$, we can obtain better estimates as the next consideration shows. Using Lemma \ref{lem-2} and the triangle's inequality we have
\begin{align*}
F_{\lambda,\mu}(f)&=\left|\left(a_{3}-a_{2}^{2}\right)+(1-\lambda)a_{2}^{2}\right|-\mu|a_{2}|\geq
|1-\lambda||a_{2}|^{2}-\left|a_{3}-a_{2}^{2}\right|-\mu|a_{2}|\\
&\geq|1-\lambda||a_{2}|^{2}-\frac{1}{3}\left(1-|a_{2}|^{2}\right)-\mu|a_{2}|=\left(|1-\lambda|+\frac{1}{3}\right)|a_{2}|^{2}-\mu|a_{2}|-\frac{1}{3}\\
&=:\psi_{2}\left(|a_{2}|\right),
\end{align*}
where
\begin{equation}\label{eq 25}
\psi_{2}(t)=\left(|1-\lambda|+\frac{1}{3}\right)t^{2}-\mu t-\frac{1}{3},\;0\leq t\leq1.
\end{equation}

From \eqref{eq 25} we have $\psi_{2}(0)=-\frac{1}{3}$, $\psi_{2}(1)=|1-\lambda|-\mu$, and if
\[
0\le t_{0}=\frac{\mu}{2\left(|1-\lambda|+\frac{1}{3}\right)}\le1
\]
then the function $\psi_{2}$ has its minimum equal to
\[
\psi_{2}\left(t_{0}\right)=-\frac{1}{3}-\frac{\mu^{2}}{4\left(|1-\lambda|+\frac{1}{3}\right)}.
\]
We note that $0\le t_{0}\le1$ if and only if $|1-\lambda|\ge\frac{\mu}{2}-\frac{1}{3}$, with $\mu>\frac{2}{3}$ because $\mu=\frac{2}{3}$ implies the excluded case $\lambda=1$.

In case $|1-\lambda|\leq\frac{\mu}{2}-\frac{1}{3}$, $\mu>\frac{2}{3}$, we have
\[
\min\left\{\psi_{2}(t):t\in[0,1]\right\}=\min\left\{\psi_{2}(0);\psi_{2}(1)\right\}=|1-\lambda|-\mu.
\]

For the second case it is evident that the equality holds for the function $f_{1}$ and it's rotations. Like for the first case, the equality $F_{\lambda,\mu}(f)=\psi_{2}\left(t_{0}\right)$ could be possible if and only if we have equality in \eqref{eq 19}, i.e. if $f$ has the form $f_{t_{0}}$ defined by \eqref{eq 20} for $\alpha=t_{0}$ and $t_{0}=\frac{\mu}{2\left(|1-\lambda|+\frac{1}{3}\right)}$. The condition $F_{\lambda,\mu}(f)=\psi_{2}\left(t_{0}\right)$ is equivalent to
\[
\left|\frac{1}{3}\left(2t_{0}^{2}+1\right)-\lambda t_{0}^{2}\right|-\mu t_{0}=\left(|1-\lambda|+\frac{1}{3}\right)t_{0}^{2}-\mu t_{0}-\frac{1}{3},
\]
where $|1-\lambda|\geq\frac{\mu}{2}-\frac{1}{3}$, $\mu>\frac{2}{3}$, or
\begin{equation}\label{eq 26}
\left|\left((1-\lambda)-\frac{1}{3}\right)t_{0}^{2}+\frac{1}{3}\right|=\left(|1-\lambda|+\frac{1}{3}\right)t_{0}^{2}-\frac{1}{3}.
\end{equation}

From the condition $|1-\lambda|\geq\frac{\mu}{2}-\frac{1}{3}$, $\mu>\frac{2}{3}$, let consider the case $\lambda\in\mathbb{R}$, hence $1-\lambda\leq\frac{1}{3}-\frac{\mu}{2}<0$. Under this assumption the left hand side of the equality \eqref{eq 26} becomes
\[
L=\left|\left(-|1-\lambda|-\frac{1}{3}\right)t_{0}^{2}+\frac{1}{3}\right|=\left|\left(|1-\lambda|+\frac{1}{3}\right)t_{0}^{2}-\frac{1}{3}\right|,
\]
hence the equality \eqref{eq 26} is valid if and only if $\left(|1-\lambda|+\frac{1}{3}\right)t_{0}^{2}-\frac{1}{3}\geq0$, i.e.
$|1-\lambda|\leq\frac{3}{4}\mu^{2}-\frac{1}{3}$, or since $1-\lambda<0$, whenever
\[
1-\lambda\geq\frac{1}{3}-\frac{3}{4}\mu^{2},\;\mu>\frac{2}{3},
\]
and the proof of the theorem is complete.
\end{proof}

\begin{rem}\label{remlastthm}
1. From the beginning of the Theorem \ref{thrmK2} and the obtained results, we also note that both estimates of \eqref{estimthmK2} are greater or equal than $-\mu$ under the conditions given in the statement of the theorem.

2. Let denote by
\[
\mathcal{D}_{\rho}=\left\{(\lambda,\mu)\in\mathbb{R}^{2}:\frac{1}{3}-\frac{3}{4}\mu^{2}\leq 1-\lambda\leq\frac{1}{3}-\frac{1}{2}\mu,\;\rho>\mu>\frac{2}{3}\right\},\;\frac{2}{3}<rho<+\infty,
\]
the restrictions for the parameters $\lambda$ and $\mu$ for the function $f_{t_{0}}$ that appears to the item $(i)$ of the last theorem.
Using the MAPLE\texttrademark{} computer software, in the below Figure \ref{figexth2K} appears with blue color the set $\mathcal{D}_{10}$, which shows the expected result that for big positive values of $\mu$ the range for the values of $\lambda$ become ``smaller''.
\begin{figure}[h!tb]
\includegraphics[width=0.5\textwidth]{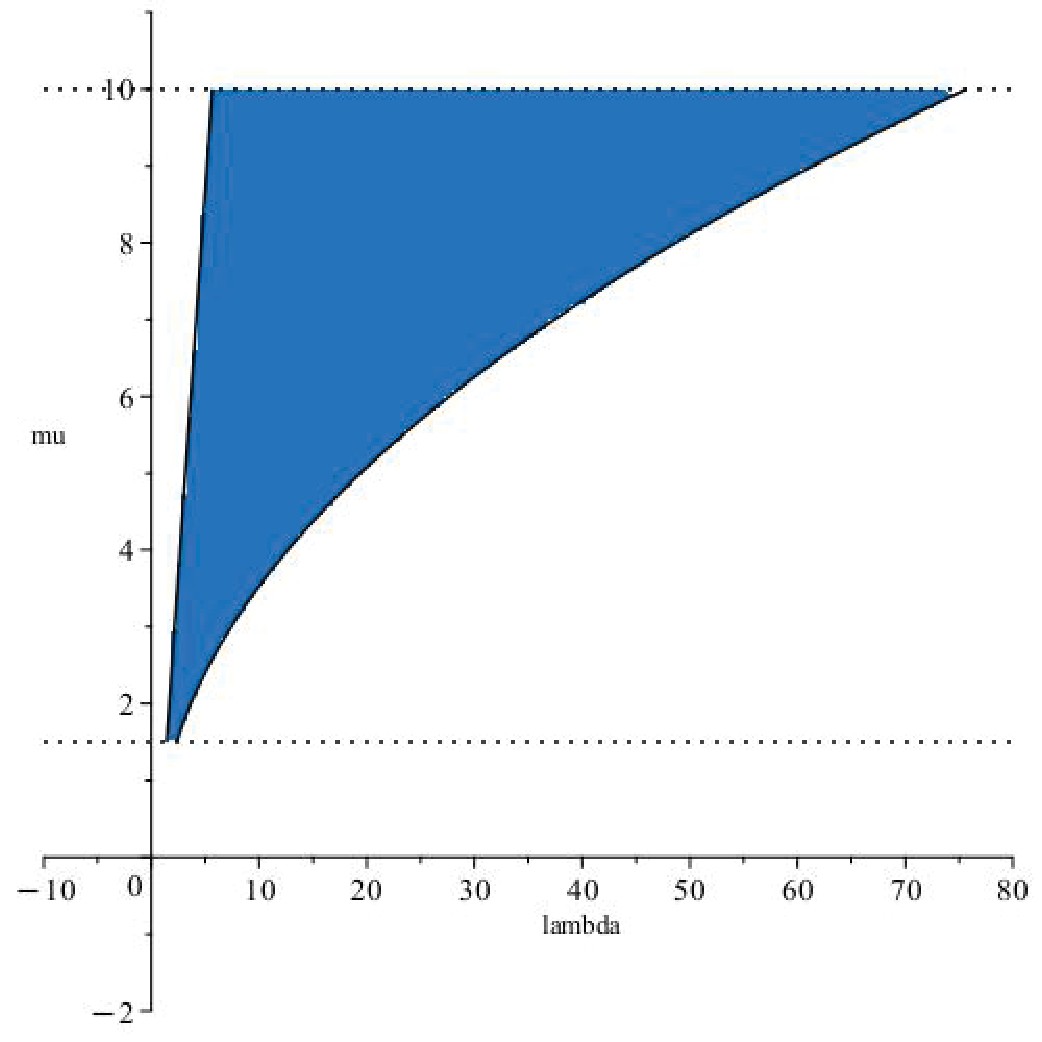}
\caption{Figure of $\mathcal{D}_{10}$ for the Remark \ref{remlastthm}}
\label{figexth2K}
\end{figure}
\end{rem}

\end{document}